\documentclass[12pt,leqno]{article}
\usepackage{amsmath, amsthm, amsfonts, amssymb, color}
\usepackage{mathrsfs}
\newtheorem{thm}{Theorem}[section]

\newtheorem{lem}[thm]{Lemma}
\newtheorem{prp}[thm]{Proposition}

\theoremstyle{definition}

\newcommand{\scr}[1]{\mathscr #1}
\definecolor{wco}{rgb}{0.5,0.2,0.3}

\numberwithin{equation}{section}

\newcommand{\ua}{\uparrow}

\date{}

\begin{document}

\def\R{\mathbb R} \def\Z{\mathbb Z} \def\ff{\frac} \def\ss{\sqrt}
\def\dd{\delta} \def\DD{\Delta} \def\vv{\varepsilon} \def\rr{\rho}
\def\<{\langle} \def\>{\rangle} \def\GG{\Gamma} \def\gg{\gamma}
\def\ll{\lambda} \def\LL{\Lambda} \def\nn{\nabla} \def\pp{\partial}
\def\d{\text{\rm{d}}} \def\loc{\text{\rm{loc}}} \def\bb{\beta} \def\aa{\alpha} \def\D{\scr D}
\def\E{\mathbb E}
\newcommand{\Ex}{{\bf E}}
\def\si{\sigma} \def\ess{\text{\rm{ess}}}
\def\beg{\begin} \def\beq{\beg}  \def\F{\scr F}
\def\Ric{\text{\rm{Ric}}} \def\Hess{\text{\rm{Hess}}}\def\B{\scr B}
\def\e{\text{\rm{e}}} \def\ua{\underline a} \def\OO{\Omega} \def\b{\mathbf b}
\def\oo{\omega}     \def\tt{\tilde} \def\Ric{\text{\rm{Ric}}}
\def\cut{\text{\rm{cut}}} \def\P{\mathbb P} \def\ifn{I_n(f^{\bigotimes n})}
\def\fff{f(x_1)\dots f(x_n)} \def\ifm{I_m(g^{\bigotimes m})} \def\ee{\varepsilon}
\def\C{\scr C}
\def\M{\scr M}\def\ll{\lambda}
\def\X{\scr X}
\def\T{\scr T}
\def\A{\scr A}
\def\LL{\scr L}
\def\gap{\mathbf{gap}}
\def\div{\text{\rm div}}
\def\dist{\text{\rm dist}}
\def\cut{\text{\rm cut}}
\def\supp{\text{\rm supp}}
\def\Var{\text{\rm Var}}
\def\p{\text{\rm p}}
\def\Cov{\text{\rm Cov}}
\def\Cut{\text{\rm Cut}}
\def\coth{\text{\rm coth}}
\def\Dom{\text{\rm Dom}}
\def\Cap{\text{\rm Cap}}
\def\Ent{\text{\rm Ent}}
\def\sect{\text{\rm sect}}\def\H{\mathbb H}

\bf STOCHASTIC EULER-POINCAR\'E REDUCTION

\rm

\vskip 10mm
MARC ARNAUDON, XIN CHEN AND ANA BELA CRUZEIRO

\vskip 15mm

Abstract

We prove a Euler-Poincar\'e reduction theorem for stochastic processes taking values in a
Lie group and we show examples of its application to $SO(3)$ and to the group of diffeomorphisms.

\section{Introduction}

Two methods of stochastic perturbation of Geometric Mechanics seem to be known today.
In one of them, inspired by J. M. Bismut [B] and  developed by J.P. Ortega  and collaborators
[LC-O],  the Lagrangian of the system is randomly perturbed. We shall advocate here the other approach,
sometimes known as  "stochastic deformation", where the Lagrangian is essentially the classical one but is evaluated on the underlying stochastic process and its mean derivative. This perspective was initially motivated by the quantization of classical systems ([C-Z], [Y 1] , [Z]) and a probabilistic version of Feynman's path integral approach. More recently ([C-C], [A-C]) the Navier-Stokes equation was derived as solution of a stochastic variational principle for this kind of Lagrangian.

We establish here a stochastic Euler-Poincar\'e reduction theorem on a general Lie group.
We describe some applications including the stochastic Lagrangian flows intrinsically associated with the Navier-Stokes equation (without external noise). In this stochastic variational program, indeed, this equation coincides with the stochastic deformation of Arnold's picture of the Euler flow as a geodesic on the Lie group of diffeomorphisms ([A]).

\section{Semimartingales in a general Lie group}

In [C-C] and [A-C],  also  inspired by [Y2], it was shown that the Navier-Stokes equation (in the two-dimensional torus and in a compact Riemannian manifold without boundary, respectively) can be viewed as the drift part of a semi-martingale which is
a critical point of the functional whose Lagrangian is given by the kinetic energy expressed via the generalized derivative.

The stochastic variational principle is formulated in the space of volume-preserving homeomorphisms, which is a (infinite dimensional) Lie group
endowed with a right-invariant metric. It extends to the Navier-Stokes equation Arnold's variational principle for the Euler equation.

In the present work we consider a stochastic variational principle  in a  general Lie group $G$, endowed with a left-invariant
(or right-invariant) metric.

The domain of our action functionals will be a set of semimartingales. As they are not
of bounded variation with respect to time, we can not use a classical derivative in time,
but will replace it by a generalized mean derivative $\frac{D}{dt}$.

In a Euclidean setting the definition of the generalized derivative (in time) corresponds
to a derivative regularized by a conditional expectation with respect to the filtration
at each time. More precisely, given by
semi-martingale $\xi(.)$ with respect to an increasing filtration $\scr{F}_{.}$ and taking
 values in
the Euclidean space (or torus)  the generalized derivative is defined by,
\begin{equation}\label{e1}
\frac{D\xi(t)}{dt}:=\lim_{\varepsilon \downarrow 0}
\E\Big[\frac{\xi(t+\varepsilon)-\xi(t)}{\varepsilon}|\scr{F}_t\Big]
\end{equation}
In particular, as the conditional expectation of the martingale part vanishes, the generalized derivative coincides with
the derivative of the bounded variation part of the semimartingale  (its drift).

 In a general manifold $M$
a martingale can only be defined after fixing a connection $\nabla$ (see [E], [I-W]). More precisely, a $M$-valued semi-martingale $\xi(.)$ is a
$\nabla$-martingale, if for each $f\in C^{\infty}(M)$,
\begin{equation*}
M^f_t:=f(\xi(t))-f(\xi(0))-\frac{1}{2}\int_0^t \text{Hess} f(\xi(s))\big(d\xi(s), d\xi(s)\big)
\end{equation*}
is a $\mathbb{R}^1$-valued local martingale with respect to the filtration
$\scr{F}_{.}$, where $\text{Hess} f(x): T_x M \times T_x M \rightarrow \mathbb{R}$ is defined by
\begin{equation}\label{e0}
\text{Hess} f(x)\Big(A_1,A_2\Big):=\tilde A_1 \tilde A_2 f -\nabla_{\tilde A_1}\tilde A_2 f, \ \forall A_1, A_2 \in T_x M,
\end{equation}
the  vector fields $\tilde A_j, \ j=1,2$  on $M$ being smooth and such that $\tilde A_i(x)=A_i$.

When $M$ is a finite dimensional manifold $\text{Hess}f=\nabla  df$ is the covariant derivative
of the (differential) tensor field $df$ by the connection $\nabla$. For an infinite dimensional Lie group  the tensor field $df$ or  $\nabla df$ does not always exist due to divergence of infinite series, but the
definition (\ref{e0})  is valid at least for smooth cylinder functions $f$. This is why we use here definition
(\ref{e0}).

So for a $M$-valued semi-martingale $\xi(.)$  it is natural to extend the definition of (\ref{e1})
to a $\nabla$-generalized derivative as follows. If for each $f\in C^{\infty}(M)$,
\begin{equation*}
N^f_t:=f(\xi(t))-f(\xi(0))-\frac{1}{2}\int_0^t \text{Hess} f(\xi(s))\big(d\xi(s), d\xi(s)\big)-
\int_0^t A(s)f (\xi(s)) ds
\end{equation*}
is a $\mathbb{R}^1$-valued local martingale, where (random) $A(t) \in T_{\xi(t)} M$ a.s., then
we define
\begin{equation}\label{e2}
\frac{D^{\nabla}\xi(t)}{dt}:=A(t)
\end{equation}

In fact, if $M$ is a finite dimensional manifold with a connection $\nabla$, there is an equivalent definition to
(\ref{e2}).  For simplicity, we assume $\xi(.)$ is a $M$-valued semi-martingale with fixed initial
point $\xi(0)=x$, there is a stochastic parallel translation $//_{.}: T_{x}M
$ $\rightarrow T_{\xi(.)}M$ along $\xi(.)$ by the connection $\nabla$ such that  $\nabla_{\circ d\xi(t)}(//_{t}v)=0$ for any
$v \in T_{x}M$.  Then $\eta(t):=\int_0^t //^{-1}_{s} \circ d\xi(s)$ is a $T_x M$ valued semimartingale. As in (\ref{e1}), we take the
derivative of bounded variation part as follows,
$$\frac{D\eta(t)}{dt}:=\lim_{\varepsilon \downarrow 0}
\E\Big[\frac{\eta(t+\varepsilon)-\eta(t)}{\varepsilon}|\scr{F}_t\Big],$$
which is a $T_x M$ valued process. Then we define
\begin{equation*}
\frac{D^{\nabla}\xi(t)}{dt}:=//_{t}\frac{D\eta(t)}{dt}.
\end{equation*}
This definition is the same as (\ref{e2}), see [E].

From now on  $G$ will denote a Lie group endowed with a left invariant metric $\langle\ \rangle$ and a left invariant connection $\nabla$.
Unless explicitely stated,
 $\nabla$ is a general connection, not necessarily the Levi-Civita connection with respect to $\langle\ \rangle$.
From now on, we let $\scr{G}:=T_{e}G$, here $e$ is the unit element of $G$, in particular, $T_e G$ can be identified with
the Lie algebra of $G$.

 Taking a sequence of vectors $H_i \in \scr{G}$, $i=1,2,..,k$, and
a non-random map $u(.) \in C^1 ([0,1];\scr{G})$, consider the following Stratonovich SDE in the group $G$,
\begin{equation}\label{e3}
dg(t)=T_{e}L_{g(t)}\Big(\sum_{i}H_i\circ dW^i_t-\frac{1}{2}\sum_{i}\nabla_{H_i}H_idt + u(t)dt\Big),
...,g(0)=e
\end{equation}
where $T_{a}L_{g(t)}:T_a G \rightarrow T_{g(t)a}G$ is the differential of the left translation $L_{g(t)}(x):=g(t)x$,
$\forall x \in G$ at the point $x=a \in G$, and $W_t$ is a $\mathbb{R}^k$ valued Brownian motion.

By It\^o formula and definition (\ref{e2}) we can see that
$$\frac{D^{\nabla}g(t)}{dt}=T_{e}L_{g(t)}u(t).$$

In fact the term $\frac{1}{2}\sum_{i}\nabla_{H_i}H_i$ corresponds to the  contraction term which
is the difference between the It\^o and the Stratonovich stochastic integral.

In particular, if $\{H_i\}$ is an orthonormal basis of $\scr{G}$, $\nabla$ is the Levi-Civita
connection, $u(t)=0$ for each $t$, and $\nabla_{H_i}H_i=0$ for each $i$, then $g(.)$ is the Brownian motion on $G$ whose generator is
the Laplace-Beltrami operator.

Note that if $H_i=0$ for each $i$, then $\frac{D^{\nabla}g(t)}{dt}$ is the ordinary derivative with
$t$, which does depend on the connection $\nabla$.

\vskip 5mm

\bf Remark: \rm  By the standard theory, the SDE (\ref{e3})  can only be defined on the finite dimensional Lie group. But in some special cases of
infinite dimensional Lie groups, for example,  the group of diffeomorphism on torus, SDE (\ref{e3})   still defines a semi-martingale even when we
take an infinite number of $H_i$, see the discussion in [A-C],[C],[C-C]. So, from now on, $G$ can be viewed as finite or infinite dimensional
Lie group as long as a solution of (\ref{e3})  exists (notably in the case of the group of diffeomorphisms on a manifold).

\section{The stochastic Euler-Poincar\'e reduction theorem in a  Lie group}

Let $\scr{S}(G)$ be the collection of all the $G$-valued semi-martingales with coefficients (both diffusion and drift) in $C^1 ([0,1];\scr{G})$. Define a stochastic action  functional $J^{\nabla,\langle\ \rangle}$ in
$\scr{S}(G)$ as following,
\begin{equation}\label{e4}
J^{\nabla,\langle\ \rangle}(\xi(.)):=\frac{1}{2}\mathbb{E}\Big[\int_0^1 \langle T_{\xi(t)}L_{\xi(t)^{-1}}\frac{D^{\nabla}\xi(t)}{dt},
T_{\xi(t)}L_{\xi(t)^{-1}}\frac{D^{\nabla}\xi(t)}{dt} \rangle dt\Big], \ \ \forall \xi(.) \in \scr{S}(G).
\end{equation}
Notice that here $T_{\xi(t)}L_{\xi(t)^{-1}}\frac{D^{\nabla}\xi(t)}{dt} \in \scr{G}$ for each $t$.

The  Lagrangian of this action functional  corresponds to the (generalized) kinetic energy. One can easily extend our
results to more general Lagrangians,
$$ L =L(\xi (.), T_{\xi(t)}L_{\xi(t)^{-1}}\frac{D^{\nabla}\xi(t)}{dt})$$

For each $v(.)\in C^1 ([0,1];\scr{G})$ satisfying
$v(0)=v(1)=0$, let $e_{\varepsilon,v}(.) \in C^1 ([0,1];G)$ be the flow generated by
$\varepsilon v(.)$ in $G$, such that
\begin{equation*}
\begin{cases}
&\frac{d}{dt}e_{\varepsilon,v}(t)=\varepsilon T_{e}L_{e_{\varepsilon,v}(t)}\dot{v}(t),\\
&e_{\varepsilon,v}(0)=e,
\end{cases}
\end{equation*}

We say a $G$-valued semi-martingale $g(.)$ is a critical point of the action functional $J^{\nabla,\langle\ \rangle}$ if
for any $v(.)\in C^1 ([0,1];\scr{G})$ satisfying
$v(0)=v(1)=0$,
\begin{equation}\label{e5}
\frac{dJ^{\nabla,\langle\ \rangle}(g_{\varepsilon,v}(.))}{d\varepsilon}\Big|_{\varepsilon=0}=0,
\end{equation}
where $g_{\varepsilon,v}(t):=g(t)e_{\varepsilon,v}(t), \ t\in [0,1]$.

For the variation of $e_{\varepsilon,v}$  the following lemma holds,
\begin{lem}\label{l1}
We have,
\begin{equation*}
\frac{d}{d\varepsilon}e_{\varepsilon,v}(t)\big|_{\varepsilon=0}=v(t)
\end{equation*}
\begin{equation*}
\frac{d}{d\varepsilon}e_{\varepsilon,v}^{-1}(t)\big|_{\varepsilon=0}=-v(t)
\end{equation*}
\end{lem}
\begin{proof}
In the proof, we omit the index $v$ in $e_{\varepsilon,v}$ for simplicity.
If $\frac{\hat{D}}{dt}$ denotes the covariant derivative on $G$ via the Levi-Civita connection, then
\begin{equation*}
\begin{split}
&\frac{\hat{D}}{dt}\frac{d}{d \varepsilon}e_{\varepsilon}(t)=
\frac{\hat{D}}{d\varepsilon}\frac{d}{dt}e_{\varepsilon}(t)=
\frac{\hat{D}}{d\varepsilon}\big(\varepsilon
T_{e}L_{e_{\varepsilon}(t)}\dot{v}(t)\big)=T_{e}L_{e_{\varepsilon}(t)}\dot{v}(t)+
\varepsilon \frac{\hat{D}}{d\varepsilon}\big(T_{e}L_{e_{\varepsilon}(t)}\dot{v}(t)\big)
\end{split}
\end{equation*}
Let $X(t):=\frac{d}{d\varepsilon}e_{\varepsilon}(t)\big|_{\varepsilon=0}$; taking $\varepsilon=0$ above, and noting that $e_{0}(t)=e$ for each $t$, we derive,
\begin{equation*}
\frac{d}{dt}X(t)=\dot{v}(t),
\end{equation*}
Then, as $v(0)=0$, we get $X(t)=v(t)$.

Since $e_{\varepsilon}(t)e_{\varepsilon}^{-1}(t)=e$ for each $\varepsilon$, differentiating with respect to $\varepsilon$, we get $\frac{d}{d\varepsilon}e_{\varepsilon}^{-1}(t)$
$=-T_{e}R_{e_{\varepsilon}^{-1}(t)}$ $T_{e_{\varepsilon}(t)}L_{e_{\varepsilon}^{-1}(t)}\frac{d e_{\varepsilon}(t)}{d\varepsilon}$,
where $TR$ is the differential of right translation. Hence we have,
\begin{equation*}
\frac{d}{d\varepsilon}e_{\varepsilon}^{-1}(t)\big|_{\varepsilon=0}=-v(t)
\end{equation*}
\end{proof}

Now we present our main result, a sufficient and necessary condition for the critical points of
$J^{\nabla,\langle\ \rangle}$.
\begin{thm}\label{t1}
Suppose that $G$ is a Lie group with left invariant metric $\langle\ \rangle$ and left invariant
connection $\nabla$. The $G$-valued semi-martingale $g(.)$ defined by (\ref{e3}) is a critical point of
$J^{\nabla,\langle\ \rangle}$ if and only if the non-random $u(.) \in C^1([0,1];\scr{G})$ satisfies the following equation,
\begin{equation}\label{e6}
\frac{d}{dt}u(t)=ad^*_{\tilde u(t)} u(t)+K(u(t)),
\end{equation}
where
\begin{equation}\label{e7a}
 \tilde u(t):= u(t)-\frac{1}{2} \sum_{i} \nabla_{H_i} H_i,
\end{equation}
for each $u \in \scr{G}$, $ad^{*}_{u}:\scr{G}$ $\rightarrow \scr{G}$ is the adjoint of $ad_{u}:\scr{G}$
$\rightarrow \scr{G}$ with respect to the metric $\langle\ \rangle$,
\begin{equation}\label{e7}
\langle ad^{*}_{u}v, \ w \rangle=\langle v,\ ad_{u}w\rangle \ \ \ \forall u,v,w \in \scr{G},
\end{equation}
and the operator $K:\scr{G}\rightarrow \scr{G}$ is defined as following,
\begin{equation}\label{e8}
\langle K(u),v \rangle=-\Big\langle u,\frac{1}{2}\sum_{i}\big(\nabla_{ad_{v}H_i}H_i+
\nabla_{H_i}(ad_{v}H_i)\big)\Big\rangle,\ \ \forall u,v \in \scr{G}.
\end{equation}
\end{thm}
\begin{proof}
In the proof, we omit the index $v$ in $e_{\varepsilon,v}(.)$ and $g_{\varepsilon,v}(.)$ for simplicity. As
$g_{\varepsilon}(t)=g(t)e_{\varepsilon}(t)$, by It\^o formula we get,
\begin{equation}\label{e8a}
\begin{split}
&dg_{\varepsilon}(t)=\sum_{i}T_{e}L_{g_{\varepsilon}(t)}H_i^{\varepsilon}(t)\circ dW^i_t
+T_{e}L_{g_{\varepsilon}(t)}\Big(Ad_{e_{\varepsilon}^{-1}(t)}\big(-\frac{1}{2}\nabla_{H_i}H_i+u(t)\big)\Big)dt\\
&+T_{e}L_{g_{\varepsilon}(t)}\big(T_{e_{\varepsilon}(t)}L_{e_{\varepsilon}^{-1}(t)}\dot{e}_{\varepsilon}(t)\big)dt,
\end{split}
\end{equation}
where $H_i^{\varepsilon}(t):=Ad_{e_{\varepsilon}^{-1}(t)}H_i$.
From the definition of $e_{\varepsilon}(t)$, we have
$T_{e_{\varepsilon}(t)}L_{e_{\varepsilon}^{-1}(t)}\dot{e}_{\varepsilon}(t)=\varepsilon \dot{v}(t)$. Then for each
$f \in C^{\infty}(G)$,
\begin{equation*}
\begin{split}
&N^f_t:=f(g_{\varepsilon}(t))-f(g_{\varepsilon}(0))-\frac{1}{2}\int_0^t
\text{Hess}f(g_{\varepsilon}(s))\big(dg_{\varepsilon}(s), dg_{\varepsilon}(s)\big)\\
&-\frac{1}{2}\sum_{i}
\int_0^t T_{e}L_{g_{\varepsilon}(s)}\big(\nabla_{H_i^{\varepsilon}(s)}H_i^{\varepsilon}(s)\big)f(g_{\varepsilon}(s))ds\\
&-\int_0^t T_{e}L_{g_{\varepsilon}(s)}\Big(Ad_{e_{\varepsilon}^{-1}(s)}\big(-\frac{1}{2}(\sum_i \nabla_{H_i}H_i)+u(s)\big)
+\varepsilon \dot{v}(s)\Big)f(g_{\varepsilon}(s))ds
\end{split}
\end{equation*}
is a local martingale.

By the definition of generalized derivative above, we get,
\begin{equation*}
\begin{split}
&T_{g_{\varepsilon}(t)}L_{g_{\varepsilon}^{-1}(t)}\frac{D^{\nabla}g_{\varepsilon}(t)}{dt}\\
&=\sum_i\frac{1}{2}\nabla_{H_i^{\varepsilon}(t)}H_i^{\varepsilon}(t)
+Ad_{e_{\varepsilon}^{-1}(t)}\big(-\frac{1}{2}(\sum_i\nabla_{H_i}H_i)+u(t)\big)+\varepsilon \dot{v}(t)
\end{split}
\end{equation*}
Using Lemma \ref{l1},
\begin{equation*}
\begin{split}
&\frac{d}{d\varepsilon}\Big(Ad_{e_{\varepsilon}^{-1}(t)}\big(-\frac{1}{2}(\sum_i\nabla_{H_i}H_i)+u(t)\big)\Big)
\Big|_{\varepsilon=0}\\
&=-ad_{v(t)}\big(-\frac{1}{2}(\sum_i\nabla_{H_i}H_i)+u(t)\big)=ad_{\big(-\frac{1}{2}(\sum_i\nabla_{H_i}H_i)+u(t)\big)}v(t)
\end{split}
\end{equation*}
Note that $H_i^{0}(t)=H_i,\ \forall t$ and by Lemma \ref{l1}, $\frac{dH_i^{\varepsilon}(t)}{d\varepsilon}\big|
_{\varepsilon=0}=-ad_{v(t)}H_i$. We obtain,
\begin{equation*}
\frac{d}{d\varepsilon}\nabla_{H_i^{\varepsilon}(t)}H_i^{\varepsilon}(t)\big|_{\varepsilon=0}
=-\nabla_{ad_{v(t)}H_i}H_i-\nabla_{H_i}(ad_{v(t)}H_i)
\end{equation*}
Recall that  $T_{g(t)}L_{g(t)^{-1}}\frac{D^{\nabla}g(t)}{dt}=u(t)$. We derive,
\begin{equation}\label{e9}
\begin{split}
&\frac{dJ^{\nabla,\langle\ \rangle}(g_{\varepsilon}(.))}{d\varepsilon}\big|_{\varepsilon=0}=
\mathbb{E}\int_0^1 \big\langle \frac{d}{d\varepsilon}\big(T_{g_{\varepsilon}(t)}L_{g_{\varepsilon}^{-1}(t)}\frac{D^{\nabla}g_{\varepsilon}(t)}{dt}\big)
\big|_{\varepsilon=0}, u(t)\big\rangle dt\\
&=\int_0^1\big\langle u(t), \dot{v}(t)+ad_{\big(-\frac{1}{2}(\sum_i\nabla_{H_i}H_i)+u(t)\big)}v(t)-\frac{1}{2}
\sum_i \big(\nabla_{ad_{v(t)}H_i}H_i+\nabla_{H_i}(ad_{v(t)}H_i)\big)\big\rangle dt\\
&=\int_0^1 \big\langle -\dot{u}(t)+ad^*_{\tilde u(t)} u(t) +K(u(t)),v(t) \big\rangle dt
\end{split}
\end{equation}
where in the last step, we used the integration by parts with respect to time and the condition $v(0)=v(1)=0$. Definitions
(\ref{e7a}), (\ref{e7}) and (\ref{e8}) were also used.

By definition, $g(.)$ is a critical point of $J^{\nabla,\langle\ \rangle}$ if and only if
$\frac{dJ^{\nabla,\langle\ \rangle}(g_{\varepsilon,v}(.))}{d\varepsilon}\big|_{\varepsilon=0}=0$ for each $v \in C^1([0,1];\scr{G})$.
Then  (\ref{e9}) implies
 equation (\ref{e6}) since $v$ is arbitrary.
\end{proof}

\vskip 5mm

\bf Remark 1. \rm If $H_i=0$ or we choose a connection such that
$\nabla_u v=0$ for any $u,v \in \scr{G}$, then $K(u)=0$ and equation (\ref{e6}) is the standard Euler-Poincar\'e equation, see for example [A-K], [M-R].

\bf Remark 2. \rm As we can deduce from the computation,  for each $\varepsilon$, the expression $T_{g_{\varepsilon}(t)}L_{g_{\varepsilon}^{-1}(t)}\frac{D^{\nabla}g_{\varepsilon}(t)}{dt}$
is non-random and does not depend on the initial point $g(0)$.

\bf Remark 3. \rm The critical equation (\ref{e6}) depends on the metric, connection and the
choice of $\{H_i\}$. The term $K(u)$ defined by (\ref{e8}) depends on the metric,
the connection and the choice of $\{H_i\}$ whereas $ad^*$ depends on the metric
only.

\bf Remark 4. \rm If $G$ is the group of  diffeomorphisms on the  torus the SDE (\ref{e3})  becomes
equation (\ref{e15}) of next section. We can check that the It\^o formula (\ref{e8a}) holds by
direct computation. Then the proof of  Theorem \ref{t1} is still valid, and the conclusion is true
in this case.

\vskip 5mm

For a Lie group $G$ with a right invariant metric and right invariant connection, we can define a  composition map $\diamond$ by
$a \diamond b:=ba,\ \forall a,b \in G$. Then the original metric and connection  are left invariant under the composition
$\diamond$ and we can also define the semi-martingale $g(.)$, the action functional $J(g(.))$ and
the perturbed semi-martingales $g_{\varepsilon,v}$ by the composition $\diamond$. For example, one can   check that
the semi-martingale $g(.)$ in
(\ref{e3}) is changed to the following,
\begin{equation}\label{e10}
dg(t)=T_{e}R_{g(t)}\Big(\sum_{i}H_i\circ dB^i_t-\frac{1}{2}\sum_{i}\nabla_{H_i}H_idt + u(t)dt\Big),
\end{equation}
where $T_{e}R_{g(t)}$ is the differential of right translation with $g(t)$ at the point $x=e$.
And the action functional $J$ in (\ref{e4}) is defined by right translation if we use the composition
$\diamond$ on $G$. We also say $g(.)$ is a critical point if
$\frac{d}{d\varepsilon}\big(J^{\nabla,\langle\ \rangle}(g_{\varepsilon,v})\big)\big|_{\varepsilon=0}=0$
for each $v\in C^1([0,1];\scr{g})$ with $v(0)=v(1)=0$. By the same procedure as above, we can derive the following theorem
on a Lie group with right invariant metric and connection.

\begin{thm}\label{t2}
Suppose that $G$ is a Lie group with right invariant metric $\langle,\rangle$ and right invariant
connection $\nabla$. The $G$-valued semi-martingale $g(.)$ defined in (\ref{e10}) is a critical point of
$J^{\nabla,\langle\ \rangle}$ if and only if $u(.) \in C^1([0,1];\scr{G})$ satisfies the following equation,
\begin{equation}\label{e11}
\frac{d}{dt}u(t)=-ad^*_{\tilde u(t)} u(t)-K(u(t)),
\end{equation}
where $\tilde u$ and
$K:\scr{G}\rightarrow \scr{G}$ are defined in  (\ref{e7a}) and  (\ref{e8}) respectively.
\end{thm}

\vskip 5mm
Under some special conditions, the operator $K(u)$ defined by (\ref{e8}) coincides with the de Rham-Hodge operator on the Lie group. More precisely  we
have the following Proposition,
\begin{prp}
Suppose that $G$ is a Lie group with right invariant metric $\langle\ \rangle$, and $\nabla$ is the (right invariant)
Levi-Civita connection with respect to $\langle\  \rangle$. If we assume that
$\nabla_{H_i}H_i=0$ for each $i$,  we have,
\begin{equation*}
K(u)=-\frac{1}{2}\sum_{i}\big(\nabla_{H_i}\nabla_{H_i}u+R(u,H_i)H_i\big),\ \ \forall u \in \scr{G},
\end{equation*}
here $R$ is the Riemanian curvature tensor with respect to $\nabla$. In particular, if $\{H_i\}$ is an orthonormal basis of
$\scr{G}$, then $K(u)=-\frac{1}{2}\square u:= -\frac{1}{2}\big(\Delta u+\text{Ric}(u)\big)$, where
$\Delta u:=\Delta U(x)|_{x=e}$ for the right invariant vector fields $U(x):=T_e R_x u,\ \forall u \in \scr{G},\ x \in G$.
\end{prp}
\begin{proof}
Note that for each $v \in \scr{G}$,
\begin{equation}\label{e12}
\begin{split}
&\nabla_{ad_{v}H_i}H_i+\nabla_{H_i}(ad_{v}H_i)\\
&=-\nabla_{[v,H_i]}H_i-\nabla_{H_i}[v,H_i]\\
&=-\nabla_{[v,H_i]}H_i-\nabla_{H_i}(\nabla_{v}H_i-\nabla_{H_i}v)\\
&=-\nabla_{[v,H_i]}H_i-\nabla_{v}\nabla_{H_i}H_i-\nabla_{[H_i,v]}H_i-R(H_i,v)H_i+\nabla_{H_i}\nabla_{H_i}v\\
&=R(v,H_i)H_i+\nabla_{H_i}\nabla_{H_i}v
\end{split}
\end{equation}
In the first step above, we used the property $ad_{v}u=-[v,u]$ for
every $u,v \in \scr{G}$ if we view $u,v$ as the right invariant vector fields on $G$. In the second step  we used the  fact that $\nabla$ is torsion free.
In the third step we used the definition of the  Riemanian curvature tensor. Finally we used the assumption
$\nabla_{H_i}H_i=0$.

Then by (\ref{e8}), for each $u,v \in \scr{G}$,
\begin{equation*}
\begin{split}
&\langle K(u),v\rangle=-\frac{1}{2}\Big\langle u, \sum_i \big(\nabla_{ad_{v}H_i}H_i+\nabla_{H_i}(ad_{v}H_i)
\big)\Big\rangle\\
&=-\frac{1}{2}\Big \langle u, \sum_i \big(R(v,H_i)H_i+\nabla_{H_i}\nabla_{H_i}v \big) \Big \rangle\\
&=-\frac{1}{2}\Big \langle \sum_i\big(\nabla_{H_i}\nabla_{H_i}u+ R(u,H_i)H_i\big),v\Big \rangle,
\end{split}
\end{equation*}
where in the last step we used the property
$\langle \nabla_{u}v, w\rangle=-\langle v,\nabla_{u} w\rangle $ for $u,v,w \in \scr{G}$
since $\nabla$ is Riemannian  with respect the metric $\langle\  \rangle$; we also used the symmetric property of
the curvature tensor $R$.

Since $v$ is arbitrary, we get,
\begin{equation*}
K(u)=-\frac{1}{2}\sum_{i}\big(\nabla_{H_i}\nabla_{H_i}u+R(u,H_i)H_i\big),
\end{equation*}
If $\{H_i\}$ is an orthonormal basis of $\scr{G}$, define the right invariant vector fields
$\tilde H_i(x):=T_e R_x H_i,\ U(x):=T_e R_x u,\ \forall x \in G$,
then $\Delta U(x)=$ $\sum_{i}\nabla^2 U(x)(\tilde H_i(x), \tilde H_i(x))$ $=\sum_{i}\big(\nabla_{\tilde H_i}$ $\nabla_{\tilde H_i}U(x)$
$-\nabla_{\nabla_{\tilde H_i}\tilde H_i}U(x)\big)$, hence
$\Delta u=\Delta U(x)|_{x=e}=\sum_i \big(\nabla_{H_i}\nabla_{H_i}u-\nabla_{H_i}H_iu\big)$
$=\sum_i \nabla_{H_i}\nabla_{H_i}u$ since $\nabla_{H_i}H_i=0$.
Also note that $\sum_{i}R(u,H_i)H_i=\text{Ric}(u)$, so we have $K(u)=-\frac{1}{2}\big(\Delta u +\text{Ric}(u)\big)$.

\end{proof}

\section {Some applications}
\subsection{The rigid body SO(3)}
To describe the motion of a rigid body, the configuration space is $G=SO(3)$, see [A-K] and [M-R].
$T_e G=so(3)$, the $3\times 3$ skew symmetric matrices.
Take a basis of $so(3)$, namely
\vskip 3mm

$E_1=\left ( \begin{array}{ccc} 0&0&0\\
0&0&-1\\0&1&0 \end{array} \right)$, $E_2=\left ( \begin{array}{ccc} 0&0&1\\
0&0&0\\-1&0&0 \end{array} \right)$, $E_3=\left ( \begin{array}{ccc} 0&-1&0\\
1&0&0\\0&0&0 \end{array} \right)$

It satisfies the following relations,
\begin{equation}\label{e13}
[E_1,E_2]=E_3,\ \ [E_2,E_3]=E_1,\ \ [E_3,E_1]=E_2.
\end{equation}

For $v\in so(3)$ with the form $v=\left ( \begin{array}{ccc} 0&-v_3&v_2\\
v_3&0&-v_1\\-v_2&v_1&0 \end{array} \right)$, $v_j \in \mathbb{R}^1,\ j=1,2,3$, we have,
$v=v_1E_1+v_2E_2+v_3E_3$. We define $\hat{v} \in \mathbb{R}^3$ to be the unique element
such that $v\eta=\hat{v}\times \eta$
for each $\eta \in \mathbb{R}^3$; in fact, it easy to check that $\hat{v}:=(v_1,v_2,v_3)$.

Take $I=(I_1,I_2,I_3)$ such that $I_j>0,\ j=1,2,3$ and  define an inner product in
$so(3)$ as follows,
\begin{equation*}
\langle v, v \rangle^{I}:=\sum_{j=1}^3I_jv_j^2, ~~~\forall v \in so(3)\ \text{with}\ \hat{v}=(v_1,v_2,v_3),
\end{equation*}
We extend $\langle, \rangle^{I}$ to $SO(3)$ by left translation, then we get a left invariant metric, which we still
write as $\langle, \rangle^{I}$. In particular, if $H_i=0$ for each $i$ in the  semi-martingale (\ref{e3}), then
$g(t)^{-1}\frac{dg(t)}{dt}=u(t)$, and $\widehat{u(t)}$ is the angular velocity vector. In the definition of
the Lagrangian in (\ref{e4}), if we choose the metric to be $\langle, \rangle^{I}$, then the Lagrangian is the kinetic energy with moment of
inertia $I$. See the discussion in [A-K], [M-R].

Let $\nabla^{I}$  be the Levi-Civita connection with respect to $\langle, \rangle^{I}$. By (\ref{e13}) and the formula for the
Levi-Civita connection, we derive,
\begin{equation}\label{e14}
\begin{split}
&\nabla^{I}_{E_1}E_1=0,\ \ \nabla^{I}_{E_1}E_2=\frac{1}{2} \big(1+\frac{I_2-I_1}{I_3}\big)E_3, \ \
\nabla^{I}_{E_2}E_1=\frac{1}{2} \big(-1+\frac{I_2-I_1}{I_3}\big)E_3\\
&\nabla^{I}_{E_2}E_2=0,\ \ \nabla_{E_2}^{I}E_3=\frac{1}{2} \big(1+\frac{I_3-I_2}{I_1}\big)E_1, \ \
\nabla_{E_3}^{I}E_2=\frac{1}{2} \big(-1+\frac{I_3-I_2}{I_1}\big)E_1\\
&\nabla_{E_3}^{I}E_3=0,\ \ \nabla_{E_3}^{I}E_1=\frac{1}{2} \big(1+\frac{I_1-I_3}{I_2}\big)E_2, \ \
\nabla_{E_1}^{I}E_3=\frac{1}{2} \big(-1+\frac{I_1-I_3}{I_2}\big)E_2\\
\end{split}
\end{equation}
Take $H_i:=\frac{1}{\sqrt{I_i}}E_i$ for $i=1,2,3$ in SDE (\ref{e3}), $\{H_i\}_{i=1}^3$ is an orthonormal basis
of $so(3)$. By (\ref{e13}) and (\ref{e14}), for each $v\in so(3)$ with $\hat{v}=(v_1,v_2,v_3)$,
\begin{equation*}
\sum_{i}\big(\nabla^{I}_{ad_{v}H_i}H_i+\nabla^{I}_{H_i}(ad_{v}H_i)\big)=\frac{1}{I_1I_2I_3}
\big((I_2-I_3)^2v_1E_1+(I_3-I_1)^2v_2E_2+(I_1-I_2)^2v_3E_3\big)
\end{equation*}
Then by (\ref{e8}), for every $u \in so(3)$ with $\hat{u}=(u_1,u_2,u_3)$,
\begin{equation*}
K(u)=-\frac{1}{2}\frac{1}{I_1I_2I_3}
\big((I_2-I_3)^2u_1E_1+(I_3-I_1)^2u_2E_2+(I_1-I_2)^2u_3E_3\big)
\end{equation*}
From [M-R], we know for each $u \in so(3)$ with $\hat{u}=(u_1,u_2,u_3)$, the adjoint
of $ad$ with respect to $\langle \ \rangle^{I}$ has the following expression,
\begin{equation*}
ad^*_u (u)=\frac{u_2u_3(I_2-I_3)}{I_1}E_1+\frac{u_3u_1(I_3-I_1)}{I_2}E_2+
\frac{u_1u_2(I_1-I_2)}{I_3}E_3.
\end{equation*}
Replacing in the equation (\ref{e6}), if the semi-martingale $g(.)$ in (\ref{e3}) is a critical point of
$J^{\nabla^{I},\langle\ \rangle^{I}}$, and  writting $\widehat{u(t)}=(u_1(t),u_2(t),u_3(t))$, the vector $\hat u$   satisfies the
following equation,
\begin{equation*}
\begin{cases}
& I_1\dot{u}_1(t)=(I_2-I_3)u_2(t)u_3(t)-\frac{(I_2-I_3)^2}{2I_2I_3}u_1(t)\\
& I_2\dot{u}_2(t)=(I_3-I_1)u_1(t)u_3(t)-\frac{(I_3-I_1)^2}{2I_1I_3}u_2(t)\\
& I_3\dot{u}_3(t)=(I_1-I_2)u_1(t)u_2(t)-\frac{(I_1-I_2)^2}{2I_1I_2}u_3(t)
\end{cases}
\end{equation*}
More generally, using properties (\ref{e13}) and (\ref{e14}), we can compute  equation (\ref{e6}) for the critical point
of functional $J^{\nabla^{I'},\langle\ \rangle^{I}}$ where $I,I' \in \mathbb{R}^3$ may be different. In particular, for
$I'=(1,1,1)$, by (\ref{e14}), $\nabla^{I'}_{E_i}E_j+\nabla^{I'}_{E_j}E_i=0$ for each $i,j$, which implies that
$K(u)=0$ for each $u\in so(3)$ for the metric $\langle\ \rangle^{I}$ and the connection $\nabla^{I'}$. Therefore in this case, the equation (\ref{e6}) is the same as that standard Euler-Poincar\'e
equation.

\subsection{Volume preserving diffeomorphisms on the torus }

We shall discuss the two dimensional torus $\mathbb{T}^2$ for simplicity, although
the torus of any dimension or even a more general compact Riemannian manifold can be considered as well.

Let $G^s_{V}:=\{g:=\mathbb{T}^2\rightarrow \mathbb{T}^2$ is a volume preserving
 bijection map, $g,g^{-1} \in H^s\}$, where $H^s$ is the $s$-th order Sobolev space. If $s>2$, then
 $G^s_{V}$ is an $C^{\infty}$ infinite dimensional Hilbert manifold (see [E-M]). The composition in $G^s_{V}$ will be the composition of  $\mathbb{T}^2$ maps. If $s>2$, $G_{V}^s$ is a topological group (not a Lie group since
 left translation is not smooth), see [E-M], and
 $$g_{V}^s:=T_e G_{V}^s=\{X:H^s(\mathbb{T}^2;
T\mathbb{T}^2),\ \pi(X)=e,\ divX=0\},$$
 where $e$ is the identity map between $\mathbb{T}^2$.

We consider the inner products $\langle\ \rangle^0$  and $\langle\ \rangle^1$ in $g_{V}^s$
defined as follows,
\begin{equation*}
\langle X, Y \rangle^0:=\int_{\mathbb{T}^2}\langle X(x),Y(x)\rangle_{x}dx, \ \ \forall X,Y \in g_{V}^s,
\end{equation*}
\begin{equation*}
\langle X, Y \rangle^1:=\int_{\mathbb{T}^2}\langle X(x),Y(x)\rangle_{x}dx+
\int_{\mathbb{T}^2}\langle \nabla X(x), \nabla Y(x)\rangle_{x}dx, \ \ \forall X,Y \in g_{V}^s,
\end{equation*}
where $\langle, \ \rangle$, $\nabla$ are the standard metric and corresponding Levi-Civita connection on $\mathbb{T}^2$
($\nabla$ coincides with the ordinary derivative on $\mathbb{T}^2$). We extend $\langle,\ \rangle^0$, $\langle,\ \rangle^1$
to  right invariant metrics on $G_{V}^s$ by right translation, which we still write as
$\langle,\ \rangle^0$  and $\langle,\ \rangle^1$.

By Theorem 9.1 and 9.6 in [E-M], there exists a right invariant Levi-Civita
connection $\nabla^0$ with respect to $\langle, \rangle^0$. In particular,
\begin{equation*}
\nabla_{X}^0Y=P_{e}\big(\nabla_X Y\big),\ \ \forall X,Y \in g_{V}^s,
\end{equation*}
where $\nabla$ is the Levi-Civita connection on $\mathbb{T}^2$ and $P_{e}$ is the orthogonal projection
(with respect to $L^2$) onto $g_{V}^s=$ $\{X \in H^s(T\mathbb{T}^2),\ divX=0 \}$  determined
by the Hodge decomposition, $H^s(T\mathbb{T}^2):=g_{V}^s\bigoplus d H^{s+1}(\mathbb{T}^2)$.
From now on, for $X \in g_{V}^s$  when we use $\nabla$  we view $X\in T\mathbb{T}^2$ as a vector field on
$\mathbb{T}^2$ and  when we use $\nabla^0$ we view $X$ as an element in  $g_{V}^s$.

We choose some suitable basis of $g_V^s$ as in [C-C]. We consider such basis indexed
by $k$ in a subset of $\mathbb{Z}^2$ having an unique representative of the equivalence class defined by the relation $k\simeq k'$ if $k+k' =0$. The  vectors $\{A_k,B_k\}_{k=1}^{\infty}$ will have the following form,
\begin{equation*}
A_k(\theta)=\lambda(|k|)(A_k^1(\theta), A_k^2(\theta)), \ \ \text{with}\
A_k^1(\theta)=k_2\text{cos}(k\cdot \theta),\  A_k^2(\theta)=-k_1\text{cos}(k\cdot \theta),\
\end{equation*}
\begin{equation*}
B_k(\theta)=\lambda(|k|)(B_k^1(\theta), B_k^2(\theta)), \ \ \text{with}\
B_k^1(\theta)=k_2\text{sin}(k\cdot \theta),\  B_k^2(\theta)=-k_1\text{sin}(k\cdot \theta),\
\end{equation*}
where $\theta=(\theta_1,\theta_2)\in \mathbb{T}^2$, $k=(k_1,k_2)\in \mathbb{Z}^2$,
$k \cdot \theta=k_1\theta_1+k_2\theta_2$ and  $\lambda(|k|)$  is a constant
depending only on $|k|=|k_1|+|k_2|$.
Since
$\nabla_{A_k}A_k=0,\ \ \nabla_{B_k}B_k=0\ \ \forall k$ (see the proof of Lemma 2.1 in [C-C]),  the SDE (\ref{e10}) becomes,
\begin{equation}\label{e15}
dg(t,\theta)=\sum_{k}\big(A_k(g(t,\theta))\circ dW_k^1(t)+
B_k(g(t,\theta))\circ dW_k^2(t)\big)+u(t,g(t,\theta))dt,\ \ g(0,\theta)=\theta,
\end{equation}
where $u(t,.) \in T \mathbb{T}^2, \ \forall t$, such that $u \in C^1([0,1]; g_{V}^s)$. This SDE was considered in [C-C].
If $u(t,.) \in T \mathbb{T}^2$ is regular enough and
$\lambda(|k|)$ decays to $0$ fast enough as $|k|$ tends to infinity, then the weak solution
of (\ref{e15}) exists, see [C-C]. Moreover the Stratonovich and the It\^o integrals in the equation coincide.

Note that in the proof Theorem \ref{t1}, when $\{A_k,B_k\}$ is an infinite sequence,
if $\lambda(|k|)$ decays to $0$ fast enough as $|k|$ tends to infinity, we can change the operation of
derivative with $\varepsilon$ between the operation of the infinite sum of
index $k$ and the conclusion of the Theorem is true. But for simplicity, from now on we assume
that $u(.,.)$ is smooth, and
$\{A_k,B_k\}$  is a finite sequence, i.e., there exists an integer $m>0$, such that $\lambda(|k|)=0$
for each $k$ with $|k|>m$. Furthermore, by the proof of Theorem 2.2 in [C-C], we have the following characterization,
\begin{equation}\label{e17}
\sum_{|k|\leqslant m}\big(A_kA_k f+B_kB_kf\big)=\nu \Delta f,\ \ \forall f \in C^2(\mathbb{T}^2),
\end{equation}
where $\nu:=\frac{1}{2}\sum_{k \leqslant m}\lambda^2(|k|)k_1^2$.

So the infinite dimensional Laplacian, when computed on smooth cylinder functions with only one variable, coincides with the usual Laplacian
on the torus.

\begin{prp}\label{p2}
The semi-martingale $g(,)$ in (\ref{e15}) is a critical point of the action functional
$J^{\nabla^0,\langle\ \rangle^0}$, if and only if  $u$ satisfies the following Navier-Stokes equation,
\begin{equation}\label{e18}
\begin{cases}
&\frac{\partial u}{\partial t}=-u\cdot \nabla u+\frac{\nu}{2}\Delta u+ \nabla p(t)\\
&div u=0.
\end{cases}
\end{equation}

The semi-martingale $g(,)$ in (\ref{e15}) is a critical point of the action functional  $J^{\nabla^0,\langle\ \rangle^1}$ if and only if $u$ satisfies the viscous Camassa-Holm equation,
\begin{equation}\label{e19}
\begin{cases}
&\frac{\partial v}{\partial t}=-u\cdot \nabla v-\sum_{j=1}^2 v_j \nabla u_j
+\frac{\nu}{2}\Delta v + \nabla p(t)\\
& v=u-\Delta u\\
&div u =0
\end{cases}
\end{equation}
\end{prp}
\begin{proof}
To apply Theorem \ref{t2}, we just need to give an explicit expression
of $ad^*_u (u)$ and $K(u)$ in (\ref{e7}), (\ref{e8}) for the  different metrics and connections.

For each $X \in H^s(T\mathbb{T}^2)$ and $Y \in  g_{V}^s$,
\begin{equation*}
\langle P_e X, Y\rangle^0=\int_{\mathbb{T}^2}\langle
(P_e X)(x), Y(x)\rangle dx=\int_{\mathbb{T}^2}\langle X(x),Y(x) \rangle dx.
\end{equation*}
Therefore,  for each $u, v \in g_{V}^s$ regular enough,
\begin{equation*}
\begin{split}
&\langle u, \nabla^0_{ad_{v}A_k}A_k+\nabla^0_{A_k}(ad_{v}A_k)\rangle^0=
\int_{\mathbb{T}^2}\langle u, P_e\big(\nabla_{ad_{v}A_k}A_k+\nabla_{A_k}(ad_{v}A_k)\big)\rangle dx\\
&=-\int_{\mathbb{T}^2} \langle u, (\nabla_{[v,A_k]}A_k+\nabla_{A_k}[v,A_k])\rangle dx
\end{split}
\end{equation*}
Note that $\nabla$ is the Levi-Civita connection on $\mathbb{T}^2$,
$\nabla_{A_k}A_k=0$, and the Riemanian curvature on $\mathbb{T}^2$ is zero, by the same computation in (\ref{e12})  we have,
\begin{equation*}
\nabla_{[v,A_k]}A_k+\nabla_{A_k}[v,A_k]=-\nabla_{A_k}\nabla_{A_k}v
\end{equation*}
An analogous identity holds for $B_k$, so combining the   computations above,
\begin{equation*}
\begin{split}
&\sum_{k}\langle u, \nabla^0_{ad_{v}A_k}A_k+\nabla^0_{A_k}(ad_{v}A_k)
+\nabla^0_{ad_{v}B_k}B_k+\nabla^0_{B_k}(ad_{v}B_k)\rangle^0\\
&=\sum_{k}\int_{\mathbb{T}^2}\langle u, \nabla_{A_k}\nabla_{A_k}v+\nabla_{B_k}\nabla_{B_k}v\rangle dx\\
&=\int_{\mathbb{T}^2}\langle u, \nu \Delta v\rangle dx=\int_{\mathbb{T}^2}\langle \nu \Delta u, v\rangle dx=
\langle \nu \Delta u,v\rangle^0,
\end{split}
\end{equation*}
where in the second step above we used  property (\ref{e17}), in the third step  the integration by parts formula on
$\mathbb{T}^2$, and the last step is due to the fact that  $\Delta u \in g_{V}^s$ for $u \in g_{V}^s$ regular enough. So by definition
(\ref{e8}), we have $K(u)=-\frac{\nu}{2}\Delta u$ for the metric $\langle\ \rangle^0$ and connection $\nabla^0$.

Another proof of this equality was given in [C], using the characterization of $K$ in Proposition 3.4 and
a direct computation of the operator $K$ via the computation of the Ricci tensor for the Levi-Civita connection with respect to the metric $\langle\ \rangle^0$.

From [A-K], for the metric $\langle\ \rangle^0$, we have
$ad^*_u (u)=P_e(\nabla_{u}u)=P_e(u\cdot \nabla u)$.

As a result  the reduced Euler-Poincar\'e equation
(\ref{e11}) for $J^{\nabla^0, \langle, \rangle^0}$ is the Navier-Stokes equation (\ref{e18}).

Now we consider the  metric  $\langle\ \rangle^1$.
 For each $X \in H^s(T\mathbb{T}^2)$ and $Y \in  g_{V}^s$,
\begin{equation*}
\begin{split}
&\langle P_e X, Y\rangle^1=\int_{\mathbb{T}^2}\langle
(P_e X)(x), Y(x)\rangle dx+ \int_{\mathbb{T}^2}\langle
\nabla (P_e X)(x), \nabla Y(x)\rangle dx\\
&=\int_{\mathbb{T}^2}\langle X(x),Y(x) \rangle dx+
\int_{\mathbb{T}^2}\langle
\nabla X(x), \nabla Y(x)\rangle dx,
\end{split}
\end{equation*}
Notice also that $\langle u, \Delta v \rangle^1=\langle \Delta u, v \rangle^1$ for
$u,v\in g_{V}^s$, due to the integration by parts formula on $\mathbb{T}^2$. So we can follow the same steps as we did for the
metric $\langle\ \rangle^0$ above, and obtain $K(u)=-\frac{\nu}{2}\Delta u$ for the metric $\langle\ \rangle^1$ and connection $\nabla^0$.
(The connection is still $\nabla^0$ here).

From Theorem 3.2 in [S] (note that the definition of Laplacian in [S] is the minus Laplacian here), and since $P_e(1-\Delta)^{-1}=(1-\Delta)^{-1}P_e$ on $T\mathbb{T}^2$, for the metric $\langle\ \rangle^1$, we have,
$$ad^*_u (u)=(1-\Delta)^{-1}\Big(P_e\Big(u\cdot \nabla (u-\Delta u)+\sum_{j=1}^2(u_j-\Delta u_j)\nabla u_j \Big)\Big).$$
Combing the above together, the reduced Euler-Poincar\'e equation (\ref{e11}) for
$J^{\nabla^0, \langle\  \rangle^1}$ is the viscous Cassama-Holm equation (\ref{e19}).
\end{proof}

For the standard Camassa-Holm equation we refer to [C-H] and [H-M-R], for viscous Camassa-Holm equation we refer to
[F-H-T] and [V].
\vskip 5mm

\bf Remark 1. \rm  For simplicity  we assume here that $u$ is regular enough, therefore $u$ is the classical solution of the corresponding
PDE. But to check the proof of Theorem \ref{t1}, we only need the test vector $v$ to be regular enough to go through the computation,
and under such cases, the less regular $u$ is the weak solution.

\bf Remark 2. \rm  We can  define a $H^n$ metric as $\langle X, Y\rangle^n:=$ $\int \sum_{i=0}^n \langle \nabla^i X(x), \nabla^i Y(x)\rangle dx$
for each $X,Y \in g_{V}^s$,
the corresponding critical equation (\ref{e11}) for $J^{\nabla^0,\langle\ \rangle^n}$ is as follows,
\begin{equation*}
\begin{cases}
&\frac{\partial u}{\partial t}=-ad^*_u  (u)+\frac{\nu}{2}\Delta u,\\
& div u=0,
\end{cases}
\end{equation*}
where the duality in $ad^*$ here is defined by (\ref{e7}) for metric $\langle\ \rangle^n$.

\bf Remark 3. \rm For the volume-preserving diffeomorphism group on higher dimensional torus, we can also choose an suitable basis
of the corresponding Lie algebra, see [C-M].

\vskip 10mm

REFERENCES

\vskip 5mm

[A-C] M. Arnaudon and A.B. Cruzeiro, \it Lagrangian Navier-Stokes diffusions on manifolds: variational principle and stability, \rm arXiv:1004.2176.

  [A]  V. I. Arnold, \it Sur la g\'eom\'etrie diff\'erentielle des groupes de Lie de dimension infinie et ses applications \`a l'hydrodynamique des fluides parfaits,\rm    Ann. Inst. Fourier  16  (1966),  316--361.

  [A-K] V. I. Arnold and B. Khesin, \it Topological methods in hydrodynamics, \rm Applied Math. Series 125, Springer (1998).

    [B] J. M. Bismut, \it M\'ecanique al\'eatoire, \rm Lecture Notes in Mathematics,  866, Springer
    (1981).

   [C-H] R. Camassa, D.D. Holm, \it A completely integrable dispersive shallow water equation with peaked solutions, \rm Phys. Rev. Lett. 71
  (1993), 1661--1664.

  [C-Z] K. L. Chung and J. C. Zambrini, \it  Introduction to random time and quantum randomness,
   \rm World Scientific (2003).

 [C-C] F. Cipriano and A.B. Cruzeiro, \it Navier-Stokes equation and diffusions on the group of homeomorphisms of the torus, \rm Comm. Math. Phys.  275  (2007),  no. 1, 255--269.

 [C] A.B. Cruzeiro, \it  Hydrodynamics, probability and the geometry of the diffeomorphisms group,
    \rm Seminar on Stoch. Analysis, Random Fields and Applications IV, R. C. Dalang. M. Dozzy, F. Russo ed, Birkhauser P.P. 63 (2011).

 [C-M] A.B. Cruzeiro and P. Malliavin, \it Nonergodicity of Euler fluid dynamics on tori versus positivity of the Arnold-Ricci tensor,
 \rm J. Funct. Anal.  257  no. 1, (2008), 1903--1925.

    [E] M. Emery, \it Stochastic calculus in manifolds, \rm Springer, Universitext (1989).

 [E-M]  D.G. Ebin and J.E. Marsden, \it Groups of diffeomorphisms and the motion of an incompressible fluid, \rm Ann of Math.   92  (1970),  102--163.

 [F-H-T] C. Foias, D. D. Holm and E.S. Titi, \it The three-dimensional viscous Camassa-Holm equations and their relation to the Navier-Stokes equations and turbulence theory, \rm J. Dyn. Diff. Eqns. 14, (2002), 1--35.

 [H-M-R] D.D. Holm, J.E. Marsden and T. Ratiu, \it  The Euler-Poincar\'e equations and semidirect products with applications to continuum theories,
 \rm Adv. Math. 137 no 1 (1998), 1--81.

   [I-W] N. Ikeda and S. Watanabe, \it Stochastic differential equations and diffusion processes, \rm
   North-Holland (1981).

  [LC-O] J.A. L\'azaro-Cam\'i and J.P. Ortega,\it   Stochastic Hamiltonian dynamical systems,\rm Reports
  on Math. Phys, 61 no. 1 (2008), 65--122.

  [M-R] J. E.  Marsden and T. S.  Ratiu, \it Introduction to Mechanics and Symmetry: a basic exposition
  of classical mechanical systems,\rm   Springer, Texts in Applied Math. (2003).

  [S] S. Shkoller, \it Geometry and curvature of diffeomorphism groups with H1 metric and mean hydrodynamics, \rm J. Funct. Anal. 160
  (1998), 337--365.

  [V] J. Vukadinovic, \it On the backwards behavior of the solutions of the 2D periodic viscous Camassa-Holm equations, \rm J. Dyn. Diff.
  Eqns. 14 no. 1 (2002), 37--62.

[Y1] K.  Yasue, \it Stochastic calculus of variations,\rm  Lett. Math. Phys. 4 no. 4 (1980), 357--360.

[Y2] K.  Yasue, \it A variational principle for the Navier-Stokes equation, \rm  J. Funct. Anal. 51 no. 2
(1983), 133--141.

[Z] J. C.  Zambrini, \it Variational processes and stochastic versions of mechanics,\rm    J. Math. Phys. 27 no. 9 (1986), 2307--2330.

\vskip 10mm
ACKNOWLEDGEMENTS

The authors acknowledge the financial support of the project "Probabilistic approach to finite and infinite dimensional dynamical systems" PTDC/MAT/104173/2008 from the portuguese FCT.

\vskip 10mm

\bf Marc Arnaudon
\rm

Laboratoire de Math\'ematiques et
  Applications, CNRS: UMR 7348
  Universit\'e de Poitiers, T\'el\'eport 2 - BP 30179,
  F--86962 Futuroscope Chasseneuil Cedex, France

marc.arnaudon@math.univ-poitiers.fr

\bf Xin Chen
\rm

Grupo de F\'isica-Matem\'atica Univ. Lisboa,
  Av.Prof. Gama Pinto 2
  1649-003 Lisboa, Portugal

chenxin\_217@hotmail.com

\bf Ana Bela Cruzeiro
\rm

GFMUL and Dep. de Matem\'atica IST(TUL),
  Av. Rovisco Pais
  1049-001 Lisboa, Portugal

abcruz@math.ist.utl.pt
\end{document}